\documentclass[a4paper,11pt]{amsart}

\usepackage{amsfonts,latexsym,rawfonts,amsmath,amssymb,amsthm, mathrsfs, lscape}
\usepackage[all]{xy}
\usepackage[english]{babel}
\usepackage[utf8]{inputenc}
\usepackage{xfrac}
\usepackage{multirow}

\usepackage{array, tabularx}

\usepackage{setspace}
\usepackage{textcomp}

\usepackage[hypertexnames=false,
backref=page,
    pdftex,
    pdfpagemode=UseNone,
    breaklinks=true,
    extension=pdf,
    colorlinks=true,
    linkcolor=blue,
    citecolor=blue,
    urlcolor=blue,
]{hyperref}

\usepackage[top=1.5in, bottom=1in, left=0.9in, right=0.9in]{geometry}

\def\XXint#1#2#3{{\setbox0=\hbox{$#1{#2#3}{\int}$ }
\vcenter{\hbox{$#2#3$ }}\kern-.6\wd0}}

\allowdisplaybreaks[1]

\newtheorem{theorem}{Theorem}[section]
\newtheorem{corollary}{Corollary}[theorem]
\newtheorem{lemma}[theorem]{Lemma}
\newtheorem{definition}{Definition}[theorem]
\newtheorem{proposition}{Proposition}[theorem]
\newtheorem{remark}{Remark}[theorem]

\title{Special Hermitian metrics on Oeljeklaus-Toma manifolds}

\author{Alexandra Otiman}
\address[Alexandra Otiman]{
Dipartimento di Matematica e Fisica\\
Universit\`a degli Studi Roma Tre\\
Via della Vasca Navale, 84\\
00146 Roma, Italy\\
and Institute of Mathematics ``Simion Stoilow'' of the Romanian Academy, 21,
Calea Grivitei Street, 010702, Bucharest, Romania
}
\email{aiotiman@mat.uniroma3.it, alexandra.otiman@imar.ro}

\thanks{The author is supported by GNSAGA of INdAM}


\begin{document}

\maketitle

\begin{abstract}
   Oeljeklaus-Toma (OT) manifolds are higher dimensional analogues of Inoue-Bombieri surfaces and their construction is associated to a finite extension $K$ of $\mathbb{Q}$ and a subgroup of units $U$. We characterize the existence of {\it pluriclosed} metrics (also known as {\it strongly K\" ahler with torsion (SKT)} metrics) on any OT manifold $X(K, U)$ purely in terms of number-theoretical conditions, yielding restrictions on the third Betti number $b_3$ and the Dolbeault cohomology group $H^{2,1}_{\overline{\partial}}$. Combined with the main result in \cite{dub20}, these numerical conditions render explicit examples of pluriclosed OT manifolds in arbitrary complex dimension. We prove that in complex dimension 4 and type $(2, 2)$, the existence of a pluriclosed metric on $X(K, U)$ is entirely topological, namely, it is equivalent to $b_3=2$. Moreover, we provide an explicit example of an OT manifold of complex dimension 4 carrying a pluriclosed metric. Finally, we show that no OT manifold admits {\it balanced metrics}, but all of them carry instead {\it locally conformally balanced metrics}. 
\end{abstract}

\section{Introduction}
Oeljeklaus-Toma (OT) manifolds were introduced in \cite{OT}, as a generalization in any complex dimension of Inoue-Bombieri surfaces. Their construction arises from number theory and they possess  remarkable cohomological and metric properties: they are non-K\" ahler compact complex manifolds, satisfying Hodge decomposition (\cite{aOmT}), with de Rham and Dolbeault cohomology describable in terms of number-theoretical invariants (\cite{IO}, \cite{aOmT}, \cite{k20}) and carrying a solvmanifold structure $\Gamma \backslash G$ as shown in \cite{k13}.

OT manifolds have been intensively studied especially in light of locally conformally K\" ahler (lcK) geometry, since a large subclass was shown to carry locally conformally K\" ahler metrics. Moreover, they provided examples that disproved a conjecture of Vaisman, according to which a compact lcK manifold whose Betti numbers obey the same restrictions as for a compact K\" ahler manifold must carry a K\" ahler metric. The existence of lcK metrics on OT manifolds has been translated in a number-theoretical condition, interesting per se, as we shall briefly recall in the sequel. Since the study of the Hermitian geometry of OT manifolds has been mostly reduced to the lcK condition, we shall be interested in investigating the existence of other special metrics of non-K\" ahler type ({\it pluriclosed} ({\it strongly K\" ahler with torsion or briefly SKT}), {\it balanced} and {\it locally conformally balanced metrics})  and to give an arithmetic or topological interpretation of these metric properties. The dictionary between the geometry of these manifolds and independent number-theoretical problems is part of the motivation for the problem we consider. On the other hand, the study of special Hermitian metrics has shed light on numerous interesting examples of non-K\" ahler manifolds and posed also the question of compatibility between different structures of non-K\" ahler type. In this direction, we recall a conjecture of Fino-Vezzoni according to which a compact complex manifold admitting both a pluriclosed and a balanced metric is K\" ahler. This has been already proven in specific cases in the works of Verbitsky, Chiose, Fino-Vezzoni and Fu-Li-Yau (see \cite{ver}, \cite{c14}, \cite{fv16}, \cite{fly12}). We give a complete description of the existence of pluriclosed and balanced metrics on OT manifolds and give further evidence supporting this conjecture. Before presenting the main results, we briefly recall the construction.  
\subsection{Definition of OT manifolds}
Let $\mathbb{Q} \subseteq K$ be an algebraic number field and take the $[K: \mathbb{Q}]=s+2t$ embeddings of the field $K$ in $\mathbb{C}$: the $s$ real embeddings $\sigma_1,\ldots,\sigma_s\colon K \hookrightarrow \mathbb{R}$, and the $2t$ complex embeddings $\sigma_{s+1},\ldots,\sigma_{s+t},\sigma_{s+t+1}=\overline\sigma_{s+1},\ldots,\sigma_{s+2t}=\overline\sigma_{s+t} \colon K\hookrightarrow \mathbb{C}$, that come in pairs. We shall always consider only the case when $s, t \geq 1$.

Let $\mathcal O_K$ be the ring of algebraic integers of $K$ and $\mathcal O_K^{*,+}$ the group of totally positive units. Let $\mathbb H:=\left\{z\in\mathbb{C} \mid \mathrm{Im} z>0\right\}$ denote the upper half-plane. On $\mathbb H^s\times\mathbb{C}^t$, consider the action $\mathcal{O}_K \circlearrowleft \mathbb H^s\times\mathbb{C}^t$ given by translations,
$$ T_a(w_1,\ldots,w_s,z_{1},\ldots,z_{t}) := (w_1+\sigma_1(a),\ldots,z_{t}+\sigma_{s+t}(a)) , $$
and the action $\mathcal{O}_K^{*,+} \circlearrowleft \mathbb H^s\times \mathbb{C}^t$ given by rotations,
$$ R_u(w_1,\ldots,w_s,z_{1},\ldots,z_{t}) := (w_1\cdot \sigma_1(u),\ldots,z_{t}\cdot \sigma_{s+t}(u)) . $$
Let 
\begin{equation*}
   l: \mathcal{O}^{*, +}_K \rightarrow \mathbb{R}^{s+t}
\end{equation*}
\begin{equation*}
   l(u)=\left(\mathrm{log}\,\sigma_1(u), \ldots, \mathrm{log}\,\sigma_s(u), 2\mathrm{log}\,|\sigma_{s+1}(u)|, \ldots, 2\mathrm{log}\,|\sigma_{s+t}(u)| \right).
\end{equation*}
Since $u$ is a unit, $\mathrm{Im}\, l \subseteq H:=\{(x_1, \ldots, x_{s+t}) \in \mathbb{R}^{s+t} \mid \sum^{s+t}_{i=1}x_i=0\}$. 
Oeljeklaus and Toma proved in \cite{OT} there exists $U$ a free subgroup of rank $s$ of $\mathcal{O}^{*, +}_{K}$ such that $pr_{\mathbb{R}^s} \circ l(U)$ is a lattice of rank $s$ in $\mathbb{R}^s$, where $pr_{\mathbb{R}^s}: \mathbb{R}^{s+t} \rightarrow \mathbb{R}^s$ is the projection on the first $s$ coordinates. Therefore, the action of $U \ltimes \mathcal{O}_K \circlearrowleft \mathbb{H}^s\times \mathbb{C}^t$ is fixed-point-free, properly discontinuous, and co-compact. 

The {\em OT manifold} associated to the algebraic number field $K$ and to the {\it admissible} subgroup $U$ of $\mathcal O_K^{*,+}$ is
$$X(K,U) \;:=\; \left. \mathbb H^s\times\mathbb{C}^t \middle\slash U \ltimes \mathcal O_K \right.$$
and we shall often refer to it as of type $(s, t)$. It is a compact complex manifold of dimension $n:=s+t$ and in the case $n=2$, the construction gives the Inoue-Bombieri surfaces.
\subsection{Non-K\" ahler metrics on OT manifolds}
OT manifolds have provided very interesting examples for locally conformally K\" ahler geometry, as they represent among the very few examples of non-Vaisman type in complex dimension $> 2$, along with Kato manifolds and non-diagonal Hopf manifolds. We recall that a {\it locally conformally K\" ahler metric (lcK)} $\Omega$ is a Hermitian metric which satisfies $d\Omega=\theta \wedge \Omega$ for a closed one-form $\theta$. Equivalently, it can be described as a Hermitian metric for which there exists a covering of the manifold with open sets $(U_i)_i$ and smooth functions $f_i$ on $U_i$ such that $e^{-f_i}\Omega$ is a K\" ahler metric on $U_i$.  It was proven in \cite{OT} and \cite[Appendix]{dub} that the existence of an lcK metric on a generic $X(K, U)$ is equivalent to the following arithmetic condition:
\begin{equation}\label{lckcond}
    |\sigma_{s+1}(u)|= \ldots = |\sigma_{s+t}(u)|, \qquad \forall u \in U,
\end{equation}
which can independently be studied as a number theory problem (see for instance \cite{dub}, \cite{v}) and automatically reveals that all $X(K, U)$ of type $(s, 1)$ carry an lcK metric. As a matter of fact, the translation of geometrical properties of $X(K, U)$ in a number-theoretical language has been a recurrent theme in various instances of their study (see also \cite{bra} and \cite{apv16}). 
In this short note will shall diversify this study towards finding special Hermitian metrics of non-K\" ahler type and the corresponding arithmetic interpretation. Namely, we prove the following results:
\begin{theorem}\label{thmSKT}
 The following are equivalent:
\begin{itemize}
    \item $X(K, U)$ admits a pluriclosed metric
    \item $s\leq t$ and after possibly relabeling the embeddings,
    \begin{equation}\label{SKTcond}
        |\sigma_{s+i}(u)|^2\sigma_i(u)=1, \qquad \forall u \in U,\  \forall i \in \{1, \ldots, s\}, \qquad and\, \ \ |\sigma_{s+i}(u)|=1, \forall u \in U,\  \forall i>s.
    \end{equation}
\end{itemize}
\end{theorem}

\begin{theorem}
There are no balanced metrics on any OT manifold $X(K, U).$
\end{theorem}

\begin{theorem}
Any OT manifold $X(K, U)$ carries a locally conformally balanced metric.
\end{theorem}
In Section \ref{sectionSKT} we shall present the cohomological implications of Theorem \ref{thmSKT}, as well as the particular case of complex dimension 4, where the two equivalent statements in Theorem \ref{thmSKT} can be reformulated as a topological condition. Moreover, in Subsection \eqref{example}, we give an explicit example of OT manifold of complex dimension 4 carrying a pluriclosed metric, which was communicated to me by Matei Toma.  The recent work \cite{dub20} shows that in arbitrary complex dimension, there exists an admissible pair $(K, U)$ satisfying condition \eqref{SKTcond}, thus providing examples of pluriclosed OT manifolds in any even complex dimension.
They represent rather exotic examples, expanding the list of examples known so far in the literature, which are specific cases of nilmanifolds (see \cite{fn}), solvmanifolds in complex dimension 3 (see \cite{fou}), connected sum of certain product of spheres (see \cite{ggp}), compact Lie groups (see \cite{sst} and \cite{ms} for a detailed proof) and simply connected examples in arbitrary complex dimension arising from A. Swann's twist construction (see \cite{sw10}). 
 A crucial part of our proofs will be played by the solvmanifold structure that all OT manifolds carry and we present it in the next section. 

\section{Preliminaries: The solvmanifold structure}

In \cite{k13}, H. Kasuya showed that all OT manifolds $X(K, U)$ can be organized as solvmanifolds, meaning they are quotients of a  solvable Lie group $G$ to a discrete subgroup of maximal rank $\Gamma$ acting on $G$ by left multiplications. Moreover, $G$ is endowed with a complex structure $J$ which is left-invariant (i. e. invariant to left multiplications with any element $g \in G$). 
We briefly present Kasuya's construction, which consists of organizing $\mathbb{H}^s \times \mathbb{C}^t$ as a solvable Lie group and regarding $U \ltimes \mathcal{O}_K$ as a discrete subgroup. 

Since $pr_{\mathbb{R}^s} \circ l(U)$ is a lattice of rank $s$ in $\mathbb{R}^s$, there exist real numbers $b_{ki}$, $1 \leq k \leq s$, $1 \leq i \leq t$ such that for any $u \in U$:
\begin{equation*}
    2 \mathrm{log}\,|\sigma_{s+i}(u)|=\sum^s_{k=1} b_{ki}\,\mathrm{log}\, \sigma_k(u),
\end{equation*}
or equivalently,
\begin{equation*}
    |\sigma_{s+i}(u)|^2=\prod^s_{k=1} \left(\sigma_k(u)\right)^{b_{ki}}.
\end{equation*}
Moreover, there exist real numbers $c_{ki}$ for any $1 \leq k \leq s$ and $1 \leq i \leq t$ such that 
\begin{equation*}
    \sigma_{s+i}(u)=\left(\prod^s_{k=1} \left(\sigma_k(u)\right)^{\tfrac{b_{ki}}{2}} \right) e^{\mathrm{i}\sum^s_{k=1}c_{ki}\mathrm{log}\, \sigma_{k}(u)}
\end{equation*}
We shall endow $\mathbb{H}^s \times \mathbb{C}^t$ with a group structure. To this aim we define for every $1 \leq i \leq t$ the following one-dimensional representation of $\mathbb{R}^s_+$:
\begin{equation*}
    \rho_i: \mathbb{R}_{+}^s \rightarrow \mathbb{C}
\end{equation*}
\begin{equation*}
    \rho_i (x_1, \ldots, x_s):=x_1^{\tfrac{b_{1i}}{2}}\ldots x_s^{\tfrac{b_{si}}{2}}e^{\mathrm{i}\sum^s_{k=1}c_{ki}\mathrm{log}x_k}.
\end{equation*}
Then we define for any two elements in $\mathbb{H}^s \times \mathbb{C}^t$, $(w,z)=(w_1, \ldots, w_s, z_1, \ldots z_t)$ and $(w', z')=(w'_1, \ldots, w'_s, z'_1, \ldots z'_t)$:
\begin{equation*}
    (w, z) * (w', z')=(w^1, \ldots, w^s, z^1, \ldots, z^t), 
\end{equation*}
where 
\begin{align*}
    w^i & =\mathrm{Re}\, w_i+ \mathrm{Im}\, w_i \cdot \mathrm{Re}\, w'_i + \mathrm{i}\, \mathrm{Im} w_i \cdot \mathrm{Im} w'_i, \qquad 1 \leq i \leq s\\
    z^i & =z_i+ \rho_i(\mathrm{Im}\, w_1, \ldots,\mathrm{Im}\, w_s)z'_i, \qquad 1 \leq i \leq t.
\end{align*}
It is clear that left multiplication with any $(w, z)$ is a biholomorphism with respect to the standard complex structure of $\mathbb{H}^s \times \mathbb{C}^t$ and that $U \ltimes \mathcal{O}_K$ is a discrete subgroup of $\mathbb{H}^s \times \mathbb{C}^t$ by the correspondence:
\begin{equation*}
    (u, a) \mapsto \left( \sigma_1(a)+\mathrm{i}\sigma_1(u), \ldots, \sigma_s(a)+\mathrm{i}\sigma_s(u), \sigma_{s+1}(a), \ldots, \sigma_{s+t}(a) \right).
\end{equation*}
One can easily verify that the following $(1, 0)$ forms are left-invariant and represent a co-frame for $\mathfrak{g}^{1,0}$:
\begin{equation*}
\left\{
    \begin{array}{ll}
        \omega_k:=\frac{dw_k}{\mathrm{Im}\, w_k}, 1 \leq k \leq s  \\
        \gamma_i:= \left(\prod^{s}_{k=1}\left(\mathrm{Im}\, w_k\right)^{-\frac{b_{ki}}{2}}\right)e^{-\mathrm{i}\sum^s_{k=1}c_{ki}\mathrm{log}(\mathrm{Im}\, w_k)}dz_i, 1 \leq i \leq t. 
    \end{array}
\right.
\end{equation*}
Moreover, a straightforward computation gives the following equations satisfied by the co-frame:
\begin{equation}\label{ecstruct}
\left\{
    \begin{array}{ll}
        d\omega_k= \tfrac{\mathrm{i}}{2}\omega_k \wedge \overline{\omega}_k, \qquad 1 \leq k \leq s  \\
        d\gamma_i= \sum^s_{k=1} \left( \frac{\mathrm{i}}{4}b_{ki}-\frac{1}{2}c_{ki} \right)\omega_k \wedge \gamma_i + \sum^s_{k=1} \left( -\frac{\mathrm{i}}{4}b_{ki}+\frac{1}{2}c_{ki} \right)\overline{\omega}_k \wedge \gamma_i, \qquad 1 \leq i \leq t. 
    \end{array}
\right.
\end{equation}

\section{Pluriclosed metrics}\label{sectionSKT}

A Hermitian metric $\Omega$ is called pluriclosed (strongly K\" ahler with torsion) if $dd^c\Omega=0$ (see \cite{bis}), where our convention is $d^c:=-J^{-1}dJ$. Equivalently, $\Omega$ is pluriclosed if $\partial\overline{\partial}\Omega=0$.
In \cite{fkv15}, it was shown that OT manifolds of type $(s, 1)$ do not carry pluriclosed metrics. We shall prove a general statement for all OT manifolds and characterize numerically the existence of pluriclosed metrics. The result of Fino-Kasuya-Vezzoni can then easily be obtained as a corollary. We need first the following lemma due to Ugarte: 

\begin{lemma}\label{inv}
If $X(K, U)$ admits a pluriclosed metric, then it also carries a left- invariant pluriclosed metric.
\end{lemma}
\begin{proof} Following a classical result of Milnor (\cite{mil76}), if $G$ is a simply connected Lie group admitting a co-compact discrete subgroup, then it admits also a bi-invariant volume form $d\mu$. Building on the same ideas as in \cite{belgun}, by an averaging procedure using the bi-invariant volume form, the proof of \cite[Proposition 4.1]{u07} concludes that once a compact solvmanifold $\Gamma \backslash G$ carries a pluriclosed metric, it also admits a left-invariant pluriclosed metric.
\end{proof}

\begin{theorem}\label{SKT}
Let X(K, U) be any OT manifold of type $(s, t)$. The following are equivalent:
\begin{enumerate}
    \item $X(K, U)$ admits a pluriclosed metric
    \item $s \leq t$ and after possibly relabeling the embeddings, $|\sigma_{s+i}(u)|^2\sigma_i(u)=1$, for any $u \in U$, and for any $1 \leq i \leq s$ and $|\sigma_{s+i}(u)|=1$, for any $u \in U$, and any $i > s$.
\end{enumerate}
\end{theorem}

\begin{proof} 
For the implication $(2) \Rightarrow (1)$, we consider the following metric on $\mathbb{H}^s \times \mathbb{C}^t$:
\begin{equation}\label{metricaSKT}
    \tilde{\Omega}:=\sum^s_{i=1} \mathrm{i}\frac{dw_i \wedge d\overline{w}_i}{\left(\mathrm{Im} w_i\right)^2}+ \mathrm{Im} w_i \mathrm{i}dz_i \wedge d\overline{z}_i + \sum_{i>s} \mathrm{i} dz_i \wedge d\overline{z}_i,
\end{equation}
which can easily be verified as being $dd^c$-closed and $U \ltimes \mathcal{O}_K$-invariant, therefore it descends to $X(K, U)$.
For $(1) \Rightarrow (2)$, by Lemma \ref{inv} we can assume $\Omega$ is a left-invariant pluriclosed metric on $X(K, U)$. Then we can write
\begin{equation*}
    \Omega=\Omega_0 + \Omega_{01}+\Omega_1,
\end{equation*}
where 
\begin{equation*}
    \Omega_0:= \sum^{s}_{i,j=1} a_{i\overline{j}} \mathrm{i}\omega_i \wedge \overline{\omega}_j
\end{equation*}
\begin{equation*}
    \Omega_{01}:=\sum_{\substack{1 \leq i \leq s\\1 \leq j \leq t}} a_{i\overline{s+j}}\mathrm{i}\omega_i \wedge \overline{\gamma}_j+ \sum_{\substack{1 \leq i \leq t\\1 \leq j \leq s}} a_{s+i\overline{j}}\mathrm{i}\gamma_i \wedge \overline{\omega}_{j}
\end{equation*}
\begin{equation*}
    \Omega_1:= \sum^t_{i, j=1} a_{s+i\overline{s+j}} \mathrm{i} \gamma_i \wedge \overline{\gamma}_j,
\end{equation*}
and $\left(a_{i\overline{j}}\right)_{1 \leq i, j \leq n}$ is a positive Hermitian matrix. Using now the equations \eqref{ecstruct}, we easily get that for any $1 \leq i \leq t$,

\begin{equation*}
i_{\gamma_i^*}i_{\overline{\gamma}^*_i}dJd\Omega_0=0,
\end{equation*}
\begin{equation*}
i_{\gamma_i^*}i_{\overline{\gamma}^*_i}dJd\Omega_{01}=0,
\end{equation*}
and therefore, 
\begin{equation*}
    i_{\gamma_i^*}i_{\overline{\gamma}^*_i}dJd\Omega = i_{\gamma_i^*}i_{\overline{\gamma}^*_i}dJd\Omega_{1}=-a_{s+i\overline{s+i}}\sum^s_{k=1}b_{ki}(b_{ki}+1)\omega_k \wedge \overline{\omega}_k -a_{s+i\overline{s+i}}\sum^s_{k\neq l}b_{ki}b_{li}\omega_l \wedge \overline{\omega}_k.
\end{equation*}
Since $dJd\Omega=0$ and $a_{s+i\overline{s+i}}>0$, we deduce that for any $1 \leq k \leq s$ and $1 \leq i \leq t$, we have $b_{ki} \in \{0, -1\}$ and $b_{ki}b_{li}=0$, for any $k \neq l$. However, by the construction of $X(K, U)$, $(b_{ki})_{ki}$ satisfy 
\begin{equation}
    \sum^t_{i=1}b_{ki}=-1, \qquad \forall 1 \leq k \leq s,
\end{equation}
which combined with $b_{ki} \in \{0, -1\}$ and $b_{ki}b_{li}=0$, for any $k \neq l$ gives that for any $1 \leq k \leq s$, there exists $1 \leq i_k \leq t$ such that $b_{ki_k}=-1$, $b_{ki}=0$, for $i \neq i_k$, and $i_k \neq i_l$ for $k \neq l$. In terms of embeddings, this translates in the following equality valid for any $u \in U$:
\begin{equation*}
    \sigma_{k}(u)|\sigma_{s+i_k}(u)|^2=1, \qquad 1 \leq k \leq s, \qquad |\sigma_{s+j}(u)|=1, \forall j \neq i_k. 
\end{equation*}
Finally, after a relabeling, we have $\sigma_i(u)|\sigma_{s+i}(u)|^2=1$, for any $1 \leq i \leq s$ and $|\sigma_{s+i}(u)|=1$, for any $i > s$.
\end{proof}

\begin{proposition}\label{cohomology}
An Oeljeklaus-Toma manifold $X(K, U)$ of type $(s, s)$ admitting a pluriclosed metric has the following topological and complex properties:
\begin{enumerate}
    \item The third Betti number $b_3(X(K, U))= {s \choose 3} + s$.
    \item $\mathrm{dim}_{\mathbb{C}}H^{2,1}_{\overline{\partial}}(X)=s$
\end{enumerate}

\end{proposition}

\begin{proof}
(1) The third Betti number can be computed via Theorem 3.1 in \cite{IO}, namely:
\begin{equation}\label{betti}
    b_3= {s \choose 3} + s\rho_2 + \rho_3,
\end{equation}
where $\rho_2$ is the cardinal of the set $\{1 \leq i_1 \neq i_2 \leq s+2t \mid \sigma_{i_1}(u)\sigma_{i_2}(u)=1, \forall u \in U\}$ and $\rho_3$ is the cardinal of the set 
\begin{equation*}
    \{ 1\leq i_1, i_2, i_3\leq s+2t, i_1 \neq i_2, i_2 \neq i_3, i_3 \neq i_1 \mid \sigma_{i_1}(u)\sigma_{i_2}(u)\sigma_{i_3}(u)=1, \forall u \in U\}.
\end{equation*}
Firstly we note that $\rho_2=0$ (see also \cite[Proposition 2.4]{apv16}). Indeed, it is clearly impossible to have $\sigma_{i_1}(u)\sigma_{i_2}(u)=1$, for any $u \in U$, if $1 \leq i_1 \leq s$ and $s+1 \leq i_2 \leq 3s$, or if $1 \leq i_1, i_2 \leq s$. If $s+1 \leq i_1, i_2 \leq 3s$, then also $\overline{\sigma}_{i_1}(u)\overline{\sigma}_{i_2}(u)=1$. Using now the pluriclosed condition \eqref{SKTcond}, we get 
\begin{equation}\label{ecc}
    \sigma_{\overline{i}_1}(u)\sigma_{\overline{i}_2}(u)=1, \forall u \in U,
\end{equation}
where $\overline{i}_{1, 2}=i_{1, 2}-s$, if $i_{1, 2}\leq 2s$ and $\overline{i}_{1, 2}=i_{1, 2}-2s$, if $i_{1, 2}> 2s$. But having a nontrivial relation between $\sigma_1(u), \ldots, \sigma_s(u)$ as \eqref{ecc} would imply is impossible, therefore, $\rho_2$ has to vanish. Now we prove that $\rho_3=s$. We clearly have $\rho_3 \geq s$, since for every $1 \leq i \leq s$, $\sigma_i(u)\sigma_{s+i}(u)\sigma_{2s+i}(u)=1$, for any $u \in U$. If there existed any triple $i_1 < i_2 < i_3$ such that $\sigma_{i_1}(u) \sigma_{i_2}(u)\sigma_{i_3}(u)=1$ and $(i_2, i_3) \neq (i_1+s, i_1+2s)$, using the pluriclosed condition and the fact that also $\overline{\sigma}_{i_1}(u) \overline{\sigma}_{i_2}(u)\overline{\sigma}_{i_3}(u)=1$, we obtain a non-trivial relation between $\sigma_1(u), \ldots, \sigma_s(u)$, which is impossible. Consequently, $\rho_3=s$ and $b_3={s \choose 3}+s$.

(2) By the proof of Theorem 4.5 in \cite{aOmT}, we get that in general, for any $X(K, U)$, 
\begin{equation}\label{dolbeault}
    \mathrm{dim}_{\mathbb{C}}H^{p,q}_{\overline{\partial}}=\sum_{i+j=q}{s \choose i}\sharp\{I \subseteq \{1, \ldots, s+t\}, J \subseteq \{1, \ldots, t\} \mid |I|=p, |J|=j, \sigma_{I}(u)\sigma_{\overline{J}}(u) \equiv 1\},
\end{equation}
where for a multi-index $I=(i_1, \ldots, i_k)$, we use the notation $\sigma_I(u):=\sigma_{i_1}(u)\cdot \ldots \cdot \sigma_{i_k}(u)$ and by $\sigma_{\overline{I}}(u):=\sigma_{s+t+i_1}(u)\cdot \ldots \cdot \sigma_{s+t+i_k}(u)$. Using the pluriclosed condition and the fact that there is no non-trivial relation between $\sigma_1(u), \ldots, \sigma_s(u)$, we easily obtain $\mathrm{dim}_{\mathbb{C}}H^{2,1}_{\overline{\partial}}(X)=s$.
\end{proof}

\begin{corollary}
An Oeljeklaus-Toma manifolds $X(K, U)$ admits both a pluriclosed and an lcK metric if and only if it is an Inoue-Bombieri surface.
\end{corollary}

\begin{proof}
The pluriclosed and lcK conditions (see \eqref{SKTcond} and \eqref{lckcond}) combined amount to 
\begin{equation*}
    \sigma_1(u)=\ldots=\sigma_s(u), \qquad \forall u \in U,
\end{equation*}
which can be true only if $s=1$. The existence of pluriclosed metrics implies $s=t$ by Theorem \ref{SKT}, therefore $\mathrm{dim}_{\mathbb{C}}X(K, U)=2$. In this case the metric
\begin{equation*}
    \Omega=\mathrm{i}\frac{dw \wedge d\overline{w}}{\left(\mathrm{Im}\,w\right)^2}+\mathrm{Im}\,w \mathrm{i}dz\wedge d\overline{z} 
\end{equation*}
is both lcK and pluriclosed, since on complex surfaces, the notions of pluriclosed and Gauduchon metric coincide.
\end{proof}
\begin{remark}
We expect in general a compact complex manifold of dimension $>2$, admitting both a pluriclosed and an lcK metric, to be K\" ahler.  In \cite{oos} we proved this was the case for complex compact nilmanifolds. 
\end{remark}
\begin{proposition}
Let $X(K, U)$ be an OT manifold of complex dimension 4 of type $(2, 2)$. Then the following are equivalent:
\begin{enumerate}
    \item $X(K, U)$ admits a pluriclosed metric
    \item $b_3(X(K, U))=2$
    \item $\mathrm{dim}_{\mathbb{C}}H^{2,1}_{\overline{\partial}}(X)=2$.
\end{enumerate}
\end{proposition} 
\begin{proof} 
$(1) \Rightarrow (2)$ and $(1) \Rightarrow (3)$ are already discussed in Proposition \ref{cohomology}. Assuming $b_2=2$, by \eqref{betti} and the fact that $\rho_2=0$ we get $\rho_3=2$. This automatically implies $s=t=2$ and taking into account there are no trivial relations between $\sigma_1(u)$ and $\sigma_2(u)$, we easily get that $\sigma_1(u)|\sigma_3(u)|^2\equiv 1$ and $\sigma_3(u)|\sigma_4(u)|^2\equiv 1$, after possibly relabeling the embeddings. By Theorem \ref{SKT}, $X(K, U)$ carries a pluriclosed metric and this proves $(2) \Rightarrow (1)$. Using now \eqref{dolbeault} and recycling the same arguments, we arrive again at $\sigma_1(u)|\sigma_3(u)|^2\equiv 1$ and $\sigma_3(u)|\sigma_4(u)|^2\equiv 1$, which yields the pluriclosed metric and proves $(3) \Rightarrow (1)$.
\end{proof}
\subsection{An example in complex dimension 4 of an OT manifold carrying a pluriclosed metric}\label{example}

This example was communicated to me by Matei Toma. 
We consider the following 6-degree polynomial irreducible over $\mathbb{Z}$, $f(x)=x^6+2x^3-x^2-2x+1$, which decomposes as:
\begin{equation*}
f(x)=(x^3 - \sqrt{2}x^2+(1+\sqrt{2})x-1)(x^3 + \sqrt{2}x^2+(1-\sqrt{2})x-1).
\end{equation*}
By a straightfoward analysis of the two terms, we notice that both $x^3 - \sqrt{2}x^2+(1+\sqrt{2})x-1$ and $x^3 + \sqrt{2}x^2+(1-\sqrt{2})x-1$ have one real and two non-real complex conjugated roots, therefore $f$ has two real and 4 non-real complex roots. Moreover, both real roots are in the interval $( \frac{1}{2}, 1 )$. Let us denote the real roots by $\alpha$ and $\alpha_1$ and the complex roots by $\beta, \overline{\beta}, \beta_1, \overline{\beta}_1$ and take now the number field $K=\mathbb{Q}(\alpha)$, which is a 6-degree extension of $\mathbb{Q}$. 
Then $K$ has 6 embeddings given by $\sigma_1(\alpha)=\alpha$, $\sigma_2(\alpha)=\alpha_1$, $\sigma_3(\alpha)=\beta$, $\sigma_4(\alpha)=\overline{\beta}$, $\sigma_5(\alpha)=\beta_1$, $\sigma_6(\alpha)=\overline{\beta}_1$. We notice that since $\prod^6_{i=1}\sigma_i(\alpha)=1$, $\alpha$ is a unit and moreover, $\sigma_1(\alpha)\sigma_3(\alpha)\sigma_4(\alpha)=1$ and $\sigma_2(\alpha)\sigma_5(\alpha)\sigma_6(\alpha)=1$. We claim that $1-\alpha$ is a unit as well. Indeed, this is immediate since its norm is 1 by the following reasoning:
\begin{equation*}
\prod^6_{i=1}(\sigma_i(1-\alpha))=\prod^6_{i=1}(1-\sigma_i(\alpha))=f(1)=1.
\end{equation*}
Furthermore, $(1-\sigma_1(\alpha))(1-\sigma_3(\alpha))(1-\sigma_4(\alpha))=1$ and $(1-\sigma_2(\alpha))(1-\sigma_5(\alpha))(1-\sigma_6(\alpha))=1$. 

The group of positive units $U$ generated by $\alpha$ and $1-\alpha$ is admissible for $K$, since $(\mathrm{log}\, \sigma_1(\alpha), \mathrm{log}\, \sigma_2(\alpha))$ and $(\mathrm{log}\, (1-\sigma_1(\alpha)), \mathrm{log}\, (1-\sigma_2(\alpha)))$ are linearly independent over $\mathbb{R}$. Indeed, if they were not linearly independent, there would exist $C \in \mathbb{R}$ such that
\begin{equation*}
C=\frac{\mathrm{log}\, (1-\alpha)}{\mathrm{log}\, \alpha}=\frac{\mathrm{log}\, (1-\alpha_1)}{\mathrm{log}\, \alpha_1}.
\end{equation*}
However, the function $x \mapsto \frac{\mathrm{log}\, (1-x)}{\mathrm{log}\, x}$ is strictly increasing on $(0, 1)$, therefore, $\alpha=\alpha_1$, but this is not possible. Consequently, $\alpha$ and $1-\alpha$ generate a rank 2 subgroup $U$ of $\mathcal{O}^{*, +}_{K}$, satisfying the admissibility property for $K$ and verifying also the pluriclosed condition. Thus, $X(K, U)$ provides the first example of OT manifold endowed with a pluriclosed metric.

\begin{remark}
	In the recent work \cite{dub20} it was shown that for any $s \geq 2$, there exists a number field $K$ of type $(s, s)$ and an admissible positive units group $U$ satisfying the pluriclosed condition. Therefore, OT manifolds provide examples of pluriclosed metrics in arbitrary even complex dimension. 
\end{remark}
\section{Balanced metrics}

A Hermitian metric $\Omega$ is called balanced if $d\Omega^{n-1}=0$ (see \cite{Mic}), or equivalently, if $d^*\Omega=0$, where $d^*$ is the metric adjoint operator of $d$. 

We shall need the following result due to A. Fino and G. Grantcharov (see \cite[Theorem 2.2]{fg}) to prove the main result of the section:

\begin{lemma}\label{invbal}
Let $\Gamma \backslash G$ be a compact solvmanifold. If it carries a balanced metric, then it also admits a left-invariant balanced metric.
\end{lemma}

\begin{theorem}
There are no balanced metrics on any Oeljeklaus-Toma manifold $X(K, U)$.
\end{theorem}

\begin{proof}
We shall prove there are no closed $(n-1, n-1)$-positive forms on $X(K, U)$, which according to \cite[Lemma 4.8]{Mic}, is the same with a balanced metric. Then, by Lemma \ref{invbal}, it is sufficient to prove there are no closed left-invariant $(n-1, n-1)$-positive forms on $X(K, U)$. Let us assume $\Omega_0$ is a left-invariant $(n-1, n-1)$-positive form, then,
\begin{equation*}
    \Omega_0=\mathrm{i}^{n-1}\sum^{n}_{i, j=1} a_{i\overline{j}}m_{i\overline{j}},
\end{equation*}
where
\begin{equation}\label{monom}
  m_{i\overline{j}}=\alpha_1 \wedge \overline{\alpha}_1 \wedge \ldots \widehat{\alpha_i} \wedge \overline{\alpha}_i  \wedge \ldots \wedge \alpha_j \wedge \widehat{\overline{\alpha}_j}\wedge \ldots \wedge \alpha_n \wedge \overline{\alpha}_n,
\end{equation}
the coefficients $a_{i\overline{j}}$ are complex numbers, $\alpha_i=\omega_i$, for $1 \leq i \leq s$ and $\alpha_{s+i}=\gamma_i$, for $1 \leq i \leq t$. Moreover, we have $a_{i \overline{j}}=-\overline{a_{j\overline{i}}}$ for $i \neq j$ and $a_{i\overline{i}}=\overline{a_{i\overline{i}}}>0$, for any
$1 \leq i \leq n$. The positivity of $\Omega_0$ implies the positive definiteness of the matrix $\left(\tilde{a}_{i\overline{j}}\right)_{i,j}$ given by
\begin{equation*}
\tilde{a}_{i\overline{j}}=\left\{
    \begin{array}{ll}
      \text{  }\  a_{i\overline{j}}, \text{if } i\leq j \\
      -a_{i\overline{j}}, \text{if } i>j
    \end{array}
\right.
\end{equation*}
Let us now compute $d\Omega_0$. We have
\begin{equation*}
    d\Omega_0=\sum^{n}_{i=1} c_im_{i} + \overline{c}_{i}m_{\overline{i}},
\end{equation*}
where 
\begin{equation*}
    m_{i}=\mathrm{i}^{n-1}\alpha_1 \wedge \overline{\alpha}_1 \wedge \ldots \wedge \widehat{\alpha_{i}} \wedge \overline{\alpha}_i \wedge \ldots \wedge \alpha_{n} \wedge \overline{\alpha}_n
\end{equation*}
and
\begin{equation*}
 m_{\overline{i}}=\mathrm{i}^{n-1}\alpha_1 \wedge \overline{\alpha}_1 \wedge \ldots \wedge \alpha_{i} \wedge \widehat{\overline{\alpha}_i} \wedge \ldots \wedge \alpha_{n} \wedge \overline{\alpha}_n.
\end{equation*}
We note that $m_i=\overline{m_{\overline{i}}}$. In order to compute the coefficients $c_i$, we shall employ the equations in \eqref{ecstruct}, as the following lemma shows:
\begin{lemma}
 Let $1 \leq i \leq s$. Then 
 \begin{align*}
    dm_{i \overline{j}} & =0, \qquad  1 \leq i \neq j \leq s\\
    dm_{i \overline{i}} & = \mathrm{i}\left(-\tfrac{1}{2}m_i+\tfrac{1}{2}m_{\overline{i}}\right)\\
    dm_{i \overline{j}} & = \alpha \cdot m_{\overline{j}}, \qquad \alpha \in \mathbb{C},\, s<j \leq n.
 \end{align*}
\end{lemma}
 \begin{proof}
 For the first two relations, the key part is the following equality
 \begin{equation*}
     d \left(\bigwedge^t_{i=1} \gamma_i \wedge \overline{\gamma}_i\right)=\sum^s_{k=1}\left( \left(\sum^t_{i=1}\mathrm{i}\frac{b_{ki}}{2}\right)\omega_k - \left(\sum^t_{i=1}\mathrm{i}\frac{b_{ki}}{2}\right)\overline{\omega}_k  \right) \wedge \bigwedge^t_{i=1} \gamma_i \wedge \overline{\gamma}_i
 \end{equation*}
 Using again the relation between the coefficients $b_{ki}$, which is granted by the construction, namely $\sum^t_{i=1}{b_{ki}}=-1$, for any $1 \leq k \leq s$ we get 
 \begin{equation*}
     d \left(\bigwedge^t_{i=1} \gamma_i \wedge \overline{\gamma}_i\right)=\sum^s_{k=1}\left( -\frac{\mathrm{i}}{2}\omega_k + \frac{\mathrm{i}}{2}\overline{\omega}_k  \right) \wedge \bigwedge^t_{i=1} \gamma_i \wedge \overline{\gamma}_i,
 \end{equation*}
 which combined with $d\omega_i=\frac{i}{2}\omega \wedge \overline{\omega}_i$, gives the first two equalities of the lemma.
The last equality simply follows by applying \eqref{ecstruct}, the derivation rule and noticing that for any $1\leq i \leq t$:
\begin{equation*}
  i_{\overline{\gamma}_i^*}d\omega_j=0, \qquad 1 \leq j \leq s
\end{equation*}
\begin{equation*}
    i_{\overline{\gamma}_i^*}d\gamma_j=0, \qquad 1 \leq j \leq t.
\end{equation*}
\end{proof}
 By the lemma above we conclude that
\begin{equation*}
    d\Omega_0=-\frac{\mathrm{i}}{2}\sum^s_{i=1} a_{i\overline{i}}m_{i}+\frac{\mathrm{i}}{2}\sum^s_{i=1} a_{i\overline{i}}m_{\overline{i}}+\sum^n_{i=s+1}\left(c_im_{i}+\overline{c}_im_{\overline{i}}\right).
\end{equation*}
If $\Omega_0$ was closed, then $a_{i\overline{i}}=0$, for any $1 \leq i \leq s$, but this is impossible since $\Omega_0$ is positive.
\end{proof}
We shall prove however that OT manifolds carry some other type of special metric. 
\begin{definition}
A metric $\Omega$ is called {\it locally conformally balanced} (lcb, shortly) if $d\Omega^{n-1}=\theta \wedge \Omega^{n-1}$ for a closed one form $\theta$.
\end{definition}
Note that in general, on any complex manifold and for any Hermitian metric $\Omega$, there always exists $\theta$ a real one-form such that $d\Omega^{n-1}=\theta \wedge \Omega^{n-1}$. This is called the {\it Lee form} of $\Omega$. In these terms, an {\it lcb} metric is a Hermitian metric with closed Lee form. Equivalently, $\Omega$ is lcb if and only if there exists a covering with open sets $(U_i)_{i\in I}$ of the manifold and smooth functions $f_i$ on each $U_i$ such that $e^{-f_i}\Omega$ is balanced.
\begin{theorem}\label{lcb}
Any Oeljeklaus-Toma manifold $X(K, U)$ admits a locally conformally balanced metric.
\end{theorem}
\begin{proof}
We simply take the following $(n-1, n-1)$-positive form:
\begin{equation*}
    \Omega_0:= \mathrm{i}^{n-1}\sum^n_{i=1}m_{i\overline{i}}.
\end{equation*}
Then it is a straightforward computation to see that
\begin{equation*}
d\Omega_0=-\sum^s_{i=1}\frac{\omega_i-\overline{\omega}_i}{2\mathrm{i}} \wedge \Omega_0.
\end{equation*}
Moreover, $\theta_0:=-\sum^s_{i=1}\frac{\omega_i-\overline{\omega}_i}{2\mathrm{i}}$ is a left-invariant closed real one-form. By \cite[Lemma 4.8]{Mic}, we conclude there exists a positive $(1, 1)$-form $\omega_0$ such that 
\begin{equation*}
    d\omega^{n-1}_0=\theta_0 \wedge \omega_0^{n-1},
\end{equation*}
which means precisely a locally conformally balanced metric. In the natural coordinates on $\mathbb{H}^s \times \mathbb{C}^t$, $\Omega_0$ can be written as:
\begin{equation*}
    \Omega_0:= \left(\sum^s_{i=1} \mathrm{i}^{n-1}\frac{\bigwedge_{j \neq i} dw_j \wedge d\overline{w}_j}{\prod_{j \neq i}\mathrm{Im}\,w_j^2}\right)\wedge A + B \wedge \sum^t_{i=1}\mathrm{i}^{n-1} \left(\prod^s_{k=1} \left(\mathrm{Im}\, w_k \right)^{1+b_{ki}}\right)\bigwedge_{j \neq i} dz_j \wedge d\overline{z}_j
\end{equation*}
where $A:=\left(\prod^s_{i=1}\mathrm{Im}\, w_i\right) \bigwedge^t_{i=1}dz_i \wedge d\overline{z}_i$ and $B=\frac{\bigwedge^s_{i=1} dw_i \wedge d\overline{w}_i}{\prod^s_{i=1}\mathrm{Im}\,w_i^2}$. Then $\Omega_0=\omega_0^{n-1}$, where 
\begin{equation*}
    \omega_0=\sum^s_{i=1}\frac{dw_i \wedge d\overline{w}_i}{(\mathrm{Im}\,w_i)^2}+\sum^{t}_{i=1} \prod^s_{k=1}(\mathrm{Im}\, w_k)^{-b_{ki}}dz_i \wedge d\overline{z}_i,
\end{equation*}
and this is the lcb metric presented as a $U \ltimes \mathcal{O}_K$-invariant metric on $\mathbb{H}^s \times \mathbb{C}^t$.  
\end{proof}

\begin{remark}
It was shown in \cite{v} that if $X(K, U)$ is such that $s<t$, it cannot support lcK metrics. Nevertheless, these manifolds carry locally conformally balanced metrics instead via Theorem \ref{lcb}. 
\end{remark}

\begin{remark}
Note that $\omega_0$ is also the Gauduchon metric of its conformal class. Moreover, in case the pluriclosed (lcK, respectively) condition holds, it is also a pluriclosed (lcK, respectively) metric. 
\end{remark}

\noindent{\bf Acknowledgement}: I am very grateful to Matei Toma, both for providing me with the example in Section 3.1 and for many useful suggestions and stimulating discussions that improved the paper.

\end{document}